\documentclass[preprint,12pt]{elsarticle}

\usepackage{amssymb}
\usepackage{amsthm}

\newtheorem{theorem}{Theorem}[section]
\newtheorem{lemma}[theorem]{Lemma}
\newtheorem{proposition}[theorem]{Proposition}

\newtheorem{remark}{Remark}[section]
\theoremstyle{definition}

\theoremstyle{remark}
\usepackage{amsmath}
\usepackage{graphics}
\usepackage{graphicx}
\usepackage{tikz}
\usetikzlibrary{decorations.pathreplacing}
\usepackage{lineno}
\usepackage{caption}    
\usepackage{subcaption}
\numberwithin{figure}{section}

\usepackage[
    left=2cm,
    right=2.5cm,
    top=2.5cm,
    bottom=3cm
]{geometry}
\begin{document}

\begin{frontmatter}


\title{Nonlinear stability of large amplitude viscous shock wave for compressible magnetohydrodynamics flows}
 \author[]{Lilu Sahu\fnref{label1}}
\author{T Raja Sekhar\corref{cor1}\fnref{label2}}
\ead{trajasekhar@maths.iitkgp.ac.in}
\cortext[cor1]{Corresponding author}
\affiliation{organization={$^1$Department of Mathematics},
            addressline={Indian Institute of Technology Kharagpur},
            city={Kharagpur},
            postcode={721302},
            state={West Bengal},
            country={India}}




\begin{abstract}
In this article, we establish the nonlinear time-asymptotic stability of viscous shock wave with arbitrary large amplitude to the Cauchy problem for the compressible magnetohydrodynamics system. In contrast to prior related work, which relied on various smallness assumptions on both the amplitude and the speed of the viscous shock wave, here we show that any viscous shock profile is time-asymptotically stable under small initial disturbance with zero mass condition without making any assumptions on the amplitude and speed of the viscous shock wave. Moreover, while most of the stability results are obtained for constant viscosity in the momentum equation, the present results hold even for the general bulk viscosity as a  smooth function of specific volume. The key part of the proof is introducing a new variable which provides an effective diffusion in the specific volume. 
 
\end{abstract}




\end{frontmatter}



\section{Introduction}\label{sec1}
Magnetohydrodynamics (MHD) stands as one of the most important and challenging fields from both mathematical and physical point of view. It deals with the dynamics of electrically conducting fluids, such as electrolytes, liquid metals, plasmas and their behaviour under the influence of magnetic fields.  It has significant physical relevance, with applications across a wide range of areas including geophysics, nuclear engineering, advanced cooling technologies, electronics and industrial processes such as aluminium production. This type of flow often exhibits complex mathematical structure due to the presence of strong coupling between fluid velocity and magnetic field. In this article, we consider the compressible one-dimensional MHD flows which are governed by the following system of partial differential equations \cite{chen2002global, fan2007vanishing}:
\begin{align}\label{eq1}
\begin{cases}
\rho_t + (\rho u)_x = 0,\\
(\rho u)_t + \big(\rho u^2 + p + \tfrac{1}{2}|\mathbf{b}|^2\big)_x
= (\lambda\left(\rho\right) u_x)_x, \qquad x \in \mathbb{R}, \quad t>0,\\
(\rho \mathbf{w})_t + (\rho u \mathbf{w} - \mathbf{b})_x
= (\mu \mathbf{w}_x)_x,\\
\mathbf{b}_t + (u\mathbf{b}-\mathbf{w})_x
= (\nu \mathbf{b}_x)_x ,
\end{cases}
\end{align}
where $\rho(x,t)\in \mathbb{R}^+$ denotes the fluid density, $u(x,t)\in \mathbb{R}$ represents the longitudinal velocity, $\mathbf{w}(x,t)\in \mathbb{R}^2$ describes the transverse velocity field and $\mathbf{b}(x,t)\in \mathbb{R}^2$ corresponds to the transverse magnetic field. The pressure $p(\rho)$ is a smooth function of $\rho$ as the equation of state and $\lambda(\rho)>0$ is the bulk viscosity coefficient. The constants $\mu$ and $\nu>0$ are the shear viscosity and magnetic diffusion coefficient, respectively.\\
\par In the case of one-dimensional compressible MHD system for constant bulk viscosity $\lambda>0$, we can find many interesting mathematical results in literature; for the isentropic case, Barker et al. \cite{barker2010one} show that the stability of parallel shock layers without any restriction on physical parameters, the well-posedness of strong solution locally (in time) with large initial data established in \cite{li2011local}. Moreover, Yang et al. \cite{yang2014global} established the global well-posedness of strong solutions in a bounded domain, provided the initial data are small enough, see \cite{yin2017stability,yao2023nonlinear,yao2025nonlinear} and the references therein. However, it is observed that for most of the fluids the viscosity relies on density at least in the context of the isentropic case $p(\rho) = k\rho^\gamma$ \cite{mellet2008existence,matsumura2010asymptotic,sekhar2010riemann}. Thus, in this work, we consider the viscosity coefficient $\lambda(\rho)$ as a smooth function of density. In particular, we express the governing system \eqref{eq1} in the Lagrangian mass coordinates, denote by $(x,t)$, which takes the following form \cite{ding2021asymptotic}:
\begin{align}\label{eq2}
\begin{cases}
 v_t -  u_x = 0, \\[0.3cm]
 u_t +  \left( p + \dfrac{1}{2}|\mathbf{b}|^2 \right)_x 
=   \left( \dfrac{\lambda(v)}{v}u_x \right)_x, \qquad x \in \mathbb{R}, \quad t>0, \\[0.5cm]
 \mathbf{w}_t -  \mathbf{b} _x
=   \left( \mu\dfrac{ \mathbf{w}_x}{v} \right)_x, \\[0.5cm]
(v \mathbf{b})_t -  \mathbf{w}_x
=   \left( \nu\dfrac{\mathbf{b}_x}{v} \right)_x,
\end{cases}
\end{align}
with initial data
\begin{align}\label{eq3}
    (v, u, \mathbf{w}, \mathbf{b})|_{t=0}= \left(v_0(x), u_0(x), \mathbf{w}_0(x), \mathbf{b}_0(x)\right)
\end{align}
having possibly different far field states:
\begin{align}\label{eq4}
    \left(v_0(x), u_0(x), \mathbf{w}_0(x), \mathbf{b}_0(x)\right) \longrightarrow \left(v_{\pm}, u_{\pm}, \mathbf{w}_{\pm}, \mathbf{b}_{\pm}\right), \qquad \text{as}~~x\longrightarrow \pm \infty,
\end{align}
where $v(x,t) = \frac{1}{\rho(x,t)}>0$ represents the specific volume, we consider $\lambda(v)$ to be smooth function of specific volume $v$ and the equation of state as $p = p(v)>0$ such that 
\begin{align}\label{eq5}
    p'(v)<0,\quad p''(v)> 0,\quad p''(v) \not\equiv 0
\end{align}
for all positive values of $v$. The assumptions stated on $p$ in this case are relatively normal and necessary to make sure that the related inviscid model is hyperbolic, including the isentropic case, i.e., $p(v) = kv^{-\gamma}, \gamma\geq 1,~~k>0$.
\par Observe that in the absence of dissipative effects, the inviscid MHD system arises as the idealisation of the corresponding viscous system \eqref{eq2}. Thus, it is of fundamental interest to explore the large-time stability behaviour of viscous shock solutions to the associated system \eqref{eq2}. The stability theory of viscous shock waves for the associated system was first studied by Matsumura-Nishihara \cite{matsumura1985stability} under the assumptions that the initial disturbance is small with zero mass condition. Since then, there have been significant progress of methodology and techniques in this direction, such as the weighted energy method \cite{matsumura2010asymptotic}, approximate Green's function and Evans function approach \cite{liu1997pointwise, liu2009time, kawashima1985asymptotic, mascia2003pointwise}. Matsumura-Nishihara \cite{matsumura1985stability} proved that the viscous shock profile is time-asymptotically stable provided $(\gamma-1)$(\emph{total variation of initial data}) is small and subsequently in \cite{matsumura2010asymptotic} and \cite{vasseur2016nonlinear} these assumptions are removed in the case when bulk viscosity as power of specific volume. Recently, He and Huang \cite{he2020nonlinear} achieved a further generalisation of the result by removing the power-law assumption that was made for the bulk viscosity coefficient.  In addition, some progress established in the field of viscous composite waves, which consist of a viscous shock, see \cite{huang2009stability, fan2019asymptotic, li2013asymptotic}, where the stability and large-time behaviour of such composite structures were thoroughly investigated. For more results concerning the large-time asymptotic stability of various wave patterns, see \cite{gong2026nonlinear,chhatria2025stability,liu2021nonlinear,liu2023stability} and the references therein.\\
\par It is worth noting that the above mentioned results are mainly carried out for $2\times2$ viscous gas or the full compressible Navier-Stokes equations. In contrast, the corresponding stability theory for the viscous MHD system remains comparatively less developed which motivates the present study. Nevertheless, in recent years, some progress has been made towards the stability of viscous solutions to the compressible MHD system. In particular, Yin \cite{yin2017stability} established the stability of the viscous contact discontinuity; subsequently, in \cite{yin2018stability} Yin studied the time-asymptotic stability of the composite wave pattern to the initial-boundary value problem for the inflow problem in MHD flow. Concerning the stability of the rarefaction wave, Yao and Zhu \cite{yao2023nonlinear} analysed the time-asymptotic stability under small initial perturbations and weak wave strength. Prior to this, Ding and Yin \cite{ding2021asymptotic} demonstrated the stability of the viscous shock wave to the system \eqref{eq2} under the assumption that $(\gamma-1)$(\emph{total variation of initial data}) is sufficiently small.  This assumption is similar to the assumption of Matsumura-Nishihara \cite{matsumura1985stability}. In addition, Ding and Yin required an extra condition on the shock speed, namely $s^2>\frac{1}{2}v_-$, which implies that the considered viscous shock is weak. We remark that all the above-mentioned work on MHD system was carried out for constant bulk viscosity $\lambda>0$.
\par Motivated by the progress achieved in the Navier-Stokes equations, the primary objective of this article is to extend the stability results of \cite{ding2021asymptotic} to the case of arbitrary shock amplitude for the MHD system \eqref{eq2}, allowing for a general pressure law along with bulk viscosity coefficient as a smooth function of the specific volume rather than constant. To achieve the time-asymptotic stability in a general setting, it is crucial to have sufficient diffusive structure in the nonlinear term. For this purpose, inspired by \cite{vasseur2016nonlinear, he2020nonlinear}, \cite{shelukhin1984structure}, we consider a new variable $h(x,t)$  and reduce the governing system \eqref{eq2} to a new system \eqref{eq22} where the diffusivity is now in the specific volume $v$. This reformulation plays a key role in deriving the required energy estimates and establishing the stability result.
\par The rest of the article is structured as follows: In Section \ref{sec2}, we discuss the viscous shock profile and state our main theorem. Reformulation of the governing equations in terms of the perturbation variables near the viscous shock is established in Section \ref{sec3}. Section \ref{sec4} is devoted to the derivation of a priori estimates. Finally, the main theorem is proved in Section \ref{sec5}.
\section{Viscous shock profiles and main theorem}\label{sec2}
In view of relationship between MHD and the classical Navier-Stokes equations in this article we assume that in the far field state the transverse velocity and magnetic field are zero, i.e., $\mathbf{w}_{\pm} = \mathbf{b}_{\pm} = 0$ \cite{yao2021asymptotic, yin2018stability2, yin2020convergence}. Then the large time behaviour of the Cauchy problem \eqref{eq2}-\eqref{eq4} is expect to be described by 
\begin{align}\label{eq6}
   \begin{cases}
    v_t -  u_x = 0, \\[0.2cm]
 u_t +  \left( p(v)\right)_x 
=   \left( \frac{\lambda(v)}{v}u_x \right)_x.
   \end{cases}
\end{align}
We now look for a travelling wave solution (viscous shock wave) $(\overline{v}, \overline{u})(\xi)$ to the system \eqref{eq6}, where $\xi = x-\sigma\,t$. Then $(\overline{v}, \overline{u})$ satisfies
\begin{align}\label{eq7}
    \begin{cases}
        -\sigma\overline{v}_\xi - \overline{u}_\xi = 0, \\[0.15cm]
        -\sigma\overline{u}_\xi +p(\overline{v})_\xi = \left(\frac{\lambda(\overline{v})\overline{u}_\xi}{\overline{v}}\right)_\xi, \\[0.15cm]
        (\overline{v}, \overline{u})|_{\pm \infty} = (v_\pm, u_\pm).
    \end{cases}
\end{align}
The large time behaviour of the system \eqref{eq6} is well characterised by the corresponding inviscid system 
\begin{align}
    \begin{cases}\label{eq8}
           v_t -  u_x = 0, \\[0.1cm]
           u_t +  \left( p(v)\right)_x = 0
    \end{cases}
\end{align}
with Riemann initial data
\begin{equation*}
    (v,u)(x,0) = \begin{cases}
        (v_-, u_-)\quad \text{if}~~ x<0\\
        (v_+, u_+)\quad \text{if}~~ x>0.
    \end{cases}
\end{equation*}
The eigenvalues and eigenvectors of the system \eqref{eq8} are
\begin{align*}
    \lambda_{\pm} = \pm\sqrt{-p'(v)} \quad\text{and} \quad r_{\pm}= \left(1,~ \pm\sqrt{-p'(v)}\right)^T,
\end{align*}
respectively. Since $\nabla \lambda _{\pm} \,.\, r_\pm = \pm\frac{p''(v)}{2\sqrt{-p'(v)}} $ is non-zero so, both the characteristic fields are genuinely nonlinear. As a result, the wave corresponding to both the characteristic fields are either shock or rarefaction. Henceforth, we only focus on $1-$shock ($2-$ shock is similar), i.e., shock corresponding to the first characteristic field. In this case $(v_\pm , u_\pm)$ satisfy the Rankine-Hugoniot relations
\begin{align}\label{eq9}
    \begin{cases}
        -\sigma(v_--v_+) = (u_{-}-u_+),\\
        -\sigma (u_--u_+) = (p(v_+)-p(v_-))
    \end{cases}
\end{align}
which yields
\begin{align*}
    \sigma = \pm \sqrt{-\dfrac{p(v_+)-p(v_-)}{v_+-v_-}}.
\end{align*}
Here we are concerned about $1$-shock ($\sigma<0$) and the associated entropy condition given by 
\begin{lemma}[\cite{godlewski2013numerical}]\label{lemma2.1}
 Across 1-shock, $v_->v_+$ and $u_->u_+$ if and only if Lax conditions hold, i.e.,
    \begin{align*}
        \lambda_-(v_+)< \sigma < \lambda_-(v_-),\qquad \sigma < \lambda_+(v_+).
    \end{align*}
\end{lemma}
Integrating system \eqref{eq7} over $(\pm \infty, \xi)$ and taking into account $\overline{u}_\xi \rightarrow 0$  as $\xi \rightarrow \pm \infty$ together with $\eqref{eq7}_1$ we obtain
\begin{align}\label{eq10}
    \begin{cases}
        \sigma \,\overline{v}+\overline{u} = \sigma \, v_\pm + u_\pm, \\ 
        \overline{v}_\xi = \dfrac{\overline{v}}{\sigma \lambda(v)}\left[\sigma^2 (v_- -\overline{v})+p(v_-)-p(\overline{v})\right] =: H(\overline{v})
    \end{cases}
\end{align}
For the existence of viscous shock, we have
\begin{lemma}[\cite{he2020nonlinear}]\label{lemma2.2}
     Assume that $p$ satisfies the conditions \eqref{eq5} and $\left(v_\pm, u_\pm\right)$ holds the Lax entropy conditions. Then the system \eqref{eq7} admits a unique smooth solution $(\overline{v}, \overline{u})(\xi)$ up to a shift. In addition, $v'(\xi)<0$. 
\end{lemma}
Before stating our main theorem, we now consider a new variable $h := u\,- f(v)_x$, in which 
\begin{align}\label{eq11}
    f(v) = \int^v \dfrac{\lambda(\theta)}{\theta}\text{d}\theta.
\end{align}
The variable $h$ is known as the effective velocity \cite{vasseur2016nonlinear, shelukhin1984structure}. In fact, using this effective velocity $h$, Shelukin  \cite{shelukhin1984structure} derived the required entropy estimates and in \cite{vasseur2016nonlinear,he2020nonlinear}, it was used for the case of compressible $2 \times 2$ system to study the time asymptotic stability of shock having large amplitude.\\
By \eqref{eq6}, the unknown variable $(v,h)$ satisfies
\begin{align}\label{eq12}
    \begin{cases}
        v_t - h_x = (f'(v)v_x)_x,\\
        h_t + p(v)_x = 0
    \end{cases}
\end{align}
with initial data
\begin{align}\label{eq13}
    (v,h)|_{t=0} = \left(v_0, u_0\, - f'(v_0)v_0'\right)(x) =: (v_0,h_0)(x),
\end{align}
having possibly different far-field states:
\begin{align}\label{eq14}
    \left(v, h\right)(x,t) \longrightarrow \left(v_{\pm}, u_{\pm}\right), \qquad \text{as}~~x\longrightarrow \pm \infty,~~t\geq0.
\end{align}
 Accordingly the viscous shock $(\overline{v}, \overline{h})$, where $\overline{h}:= \overline{u}-f(\overline{v})_\xi$ satisfies
 \begin{align}\label{eq15}
      \begin{cases}
        -\sigma\overline{v}_\xi - \overline{h}_\xi = \left(f'(\overline{v})\overline{v}_\xi\right)_\xi, \\[0.15cm]
        -\sigma\overline{h}_\xi +p(\overline{v})_\xi = 0, \\[0.15cm]
        (\overline{v}, \overline{h})|_{\pm \infty} = (v_\pm, u_\pm)
    \end{cases}
 \end{align}
which is equivalent to the system \eqref{eq7}. Suppose that
\begin{align}\label{eq16}
    (v_0\,-\overline{v},\, h_0-\overline{h},\, \mathbf{w}_0,\, v_0\mathbf{b}_0)(\xi)\in H^1 \cap L^1, \qquad \inf_{x\in\mathbb{R}}~v_0(x) >0
\end{align}
with
\begin{align}\label{eq17}
\begin{cases}
\displaystyle \int_{-\infty}^{+\infty} (v_0 - \overline{v})(\eta)\,\text{d}\eta = 0, 
& \displaystyle \int_{-\infty}^{+\infty} (u_0 - \overline{u})(\eta)\,\text{d}\eta = 0, \\[0.65cm]
\displaystyle \int_{-\infty}^{+\infty} \mathbf{w}_0(\eta)\,\text{d}\eta = 0, 
& \displaystyle \int_{-\infty}^{+\infty} (v_0 \mathbf{b}_0)(\eta)\,\text{d}\eta = 0.
\end{cases}
\end{align}
Further, we define the integrals
\begin{align}\label{eq18}
\begin{cases}
\displaystyle V_0(x) = \int_{-\infty}^{x} (v_0 - \overline{v})(\eta)\,\text{d}\eta, 
& \displaystyle H_0(x) = \int_{-\infty}^{x} (h_0 - \overline{h})(\eta)\,\text{d}\eta, \\[0.65cm]
\displaystyle \mathbf{W}_0(x) = \int_{-\infty}^{x} \mathbf{w}_0(\eta)\,\text{d}\eta, 
& \displaystyle \mathbf{B}_0(x) = \int_{-\infty}^{x} (v_0 \mathbf{b}_0)(\eta)\,\text{d}\eta,
\end{cases}
\end{align}
which exist for all $x\in \mathbb{R}$ and satisfy
\begin{align}\label{eq19}
    (V_0,\, H_0,\, \mathbf{W}_0,\, \mathbf{B}_0)(x) \in L^2(\mathbb{R}).
\end{align}
Here, the condition \eqref{eq17} is usually known as the zero mass condition. With the above setting our main result of the article as
\setcounter{theorem}{0}
\begin{theorem}[Main theorem]\label{theorem2.1}
 Under the assumptions \eqref{eq16}, \eqref{eq17} and \eqref{eq19}, there exists a positive constant $\delta_0$ such that if $\lVert (V_0, H_0, \mathbf{W}_0, \mathbf{B}_0) \rVert_2\leq \delta_0$, then the Cauchy problem \eqref{eq2}-\eqref{eq4} has a unique global solution $(v,u, \mathbf{w}, \mathbf{b})$ satisfying
    \begin{equation}\label{eq20}
\begin{aligned}
&(v-\overline{v}, u-\overline{u}, \mathbf{w}, v\mathbf{b}) 
    \in C^0([0,+\infty); H^1), 
    \quad u-\overline{u} \in L^2([0,+\infty); H^1), \\
&(v-\overline{v}, \mathbf{w}, v\mathbf{b}) 
    \in L^2([0,+\infty); H^2).
\end{aligned}
\end{equation}
Moreover,
\begin{align}\label{eq21}
    \sup_{x\in \mathbb{R}}\left|(v-\overline{v}, u-\overline{u}, \mathbf{w}, \mathbf{b})\right| \longrightarrow 0 \quad \text{as}~~t\longrightarrow \infty.
\end{align}
\end{theorem}
\begin{remark}
    In comparison to the results of \cite{ding2021asymptotic}, in the present work we are relaxing the conditions on shock strength and shock speed in order to study the stability of viscous shock even these results are valid for a general pressure law $p(v)$ and bulk viscosity $\lambda(v)$ as a smooth function of specific volume $v$.
\end{remark}
\noindent
\textbf{Notations:} Hereafter, $C$ represents a generic positive constant which may change from line to line, depends only upon the initial data and physical parameters but is independent of time $t$. We denote by $L^p(\mathbb{R})$, $H^s(\mathbb{R})$ the usual Lebesgue space and Sobolev space over $\mathbb{R}$ equipped with the norms $\lVert .\rVert_{L^p}$ and  $\lVert .\rVert_{H^s}$, respectively. For the sake of simplicity, $\lVert.\rVert:= \lVert.\rVert_{L^2}$ and $\lVert.\rVert_s:= \lVert.\rVert_{H^s}$.
\section{Reformulation of the problem}\label{sec3}
Now let us first rewrite the full system \eqref{eq2} in the new effective unknown variable $h$, then in order to do further analysis we consider the transformation as $(x,t)\longmapsto (\xi, t) = (x-\sigma t,\, t)$. Under this transformation the system \eqref{eq2} becomes
\begin{align}\label{eq22}
\begin{cases}
 v_t - \sigma v_\xi-h_\xi = \left(f'(v)v_\xi\right)_\xi, \\[0.3cm]
 h_t - \sigma h_\xi+\left(p(v)+\frac{1}{2}\left|\mathbf{b}\right|^2\right)_\xi 
=   0,  \\[0.3cm]
 \mathbf{w}_t - \sigma \mathbf{w}_\xi -  \mathbf{b} _\xi
=   \mu\left( \dfrac{ \mathbf{w}_\xi}{v} \right)_\xi, \\[0.3cm]
(v \mathbf{b})_t - \sigma (v \mathbf{b})_\xi -  \mathbf{w}_\xi
=   \nu\left( \dfrac{\mathbf{b}_\xi}{v} \right)_\xi,
\end{cases}
\end{align}
with initial data
\begin{align}\label{eq23}
    (v, h, \mathbf{w}, \mathbf{b})|_{t=0}= \left(v_0(x), h_0(x), \mathbf{w}_0(x), \mathbf{b}_0(x)\right).
\end{align}
In view of stability analysis, we now define the perturbation variables $(V, H, \mathbf{W}, \mathbf{B})(\xi,t)$ as
\begin{align}\label{eq24}
\begin{cases}
\displaystyle V(\xi,t) = \int_{-\infty}^{\xi} \left[v(\eta,t) - \overline{v}(\eta)\right]\,\text{d}\eta, 
& \displaystyle H(\xi,t) = \int_{-\infty}^{\xi} \left[h(\eta,t) - \overline{h}(\eta)\right]\,\text{d}\eta, \\[0.65cm]
\displaystyle \mathbf{W}(\xi,t) = \int_{-\infty}^{\xi} \mathbf{w}(\eta,t)\,\text{d}\eta, 
& \displaystyle \mathbf{B}(\xi,t) = \int_{-\infty}^{\xi} (v \mathbf{b})(\eta,t)\,\text{d}\eta.
\end{cases}
\end{align}
By virtue of \eqref{eq15}, $(V, H, \mathbf{W}, \mathbf{B})(\xi,t)$ satisfies
\begin{align}\label{eq25}
\begin{cases}
 V_t - \sigma V_\xi-H_\xi = \left(f'(v)v_\xi-f'(\bar{v})\bar{v}_\xi\right)_\xi, \\[0.3cm]
 H_t - \sigma H_\xi+\left(p(v)-p(\bar{v})+\frac{1}{2}\left|\mathbf{b}\right|^2\right)_\xi 
=   0,  \\[0.3cm]
 \mathbf{W}_t - \sigma \mathbf{W}_\xi -  \dfrac{\mathbf{B} _\xi}{v}
=   \mu \dfrac{ \mathbf{W}_{\xi \xi}}{v}, \\[0.3cm]
\mathbf{B}_t - \sigma \mathbf{B}_\xi -  \mathbf{W}_\xi
=   \nu \dfrac{1}{v}\left( \dfrac{\mathbf{B}_\xi}{v} \right)_\xi.
\end{cases}
\end{align}
We rewrite the system \eqref{eq25} into the following form
\begin{align}\label{eq26}
\begin{cases}
 V_t - \sigma V_\xi-H_\xi - f'(\bar{v})V_{\xi \xi}- f''(\bar{v})\bar{v}_\xi V_\xi = R_1, \\[0.3cm]
 H_t - \sigma H_\xi+p'(\bar{v})V_\xi
=   R_2,  \\[0.3cm]
 \mathbf{W}_t - \sigma \mathbf{W}_\xi -  \dfrac{\mathbf{B} _\xi}{\bar{v}} - \mu \dfrac{\mathbf{W}_{\xi \xi}}{\bar{v}}
=  R_3, \\[0.3cm]
\mathbf{B}_t - \sigma \mathbf{B}_\xi -  \mathbf{W}_\xi - \dfrac{\nu}{\bar{v}^2}\mathbf{B}_{\xi \xi}- \left(\dfrac{\nu}{2\bar{v}^2}\right)_\xi \mathbf{B}_{\xi}
=   R_4,
\end{cases}
\end{align}
where
\begin{equation}\label{eq27}
    \begin{aligned}
        &R_1= \left(f'(v)- f'(\bar{v})\right)V_{\xi \xi} +\left(f'(v)-f'(\bar{v}) - f''(\bar{v})V_\xi\right)\bar{v}_\xi,\\
        &R_2= -\left(p(v)-p(\bar{v})-p'(\bar{v})(v-\bar{v})\right)-\dfrac{\left|\mathbf{B}_\xi\right|^2}{2v^2},\\
        &R_3= \mu\left(\dfrac{1}{v}-\dfrac{1}{\bar{v}}\right)\mathbf{W}_{\xi \xi}+\left(\dfrac{1}{v}-\dfrac{1}{\bar{v}}\right)\mathbf{B}_{ \xi},\\
        &R_4= \nu\left(\dfrac{1}{v^2}-\dfrac{1}{\bar{v}^2}\right)\mathbf{B}_{\xi \xi}+\dfrac{\nu}{2}\left(\dfrac{1}{v^2}-\dfrac{1}{\bar{v}^2}\right)_\xi\mathbf{B}_\xi\\
    \end{aligned}
\end{equation}
with initial data, from \eqref{eq18} and \eqref{eq19}
\begin{align}\label{eq28}
    \left(V,H,\mathbf{W}, \mathbf{B}\right)|_{t=0} = \left(V_0,H_0,\mathbf{W}_0, \mathbf{B}_0\right)(\xi)\in H^2.
\end{align}
We now look for solution in the functional space $X_\delta(0, T)$, for $0\leq T<\infty$, defined by
\begin{equation}\label{eq29}
    \begin{aligned}
        X_\delta(0, T) =  \Big\{(V,H,\mathbf{W}, \mathbf{B}) &\in C\left(0,T;H^2\right)| H_\xi\in L^2\left(0,T;H^1\right),\, \left(V_\xi, \mathbf{B}_\xi, \mathbf{W}_\xi\right)\in L^2\left(0,T;H^2\right),\\
        &\sup_{0\leq t\leq T}\lVert(V,H,\mathbf{W},\mathbf{B})(t)\rVert_{2}\leq \delta \Big\}
    \end{aligned}
\end{equation}
where $\delta <<1$ is a small constant. By exploring the results from \cite{matsumura2010asymptotic, vasseur2016nonlinear}, the construction of global solution is based on local existence and \emph{a priori}  estimate followed by  continuation argument. Since by exploring contraction mapping theory we can establish the local existence, it suffices to prove the following estimates
\begin{proposition}[A priori estimate]\label{proposition3.1}
Let $(V,H, \mathbf{W}, \mathbf{B})\in X_\delta(0,T)$ be solution to the Cauchy problem \eqref{eq26}-\eqref{eq28} for some $T>0$. 
Then there exists a positive constant $\delta_0>0$, independent of $T$, such that if
\begin{equation}\label{eq30}
\sup_{t\in[0,T]}\|(V,H, \mathbf{W}, \mathbf{B})(t)\|_2 \le \delta \le \delta_0 ,
\end{equation}
then for any $t\in[0,T]$ the following estimate holds:
\begin{equation}\label{eq31}
\|(V,H, \mathbf{W}, \mathbf{B})(t)\|_2^2
+\int_0^t
\left(
\|(V_\xi, \mathbf{W}_\xi, \mathbf{B}_\xi)(\eta)\|_2^2
+
\|H_\xi(\eta)\|_1^2
\right)\, \text{d}\eta
\le
C_0\,\|(V_0, H_0, \mathbf{W}_0, \mathbf{B}_0)\|_2^2 ,
\end{equation}
where $C_0$ is a constant independent of $T$.
\end{proposition}
Once we have Proposition \ref{proposition3.1}, the following global solution holds immediately as
\setcounter{theorem}{0}
\begin{theorem}\label{theorem3.1}
Let $(V_0, H_0, \mathbf{W}_0, \mathbf{B}_0) \in H^2$. There exists a positive constant 
$\delta_1 \le \dfrac{\delta_0}{\sqrt{C_0}}$ such that if
\begin{align}\label{eq32}
    \|(V_0, H_0, \mathbf{W}_0, \mathbf{B_0})\|_2 \le \delta_1,
\end{align}
then the system \eqref{eq26}-\eqref{eq28} admits a unique global solution
\begin{align}\label{eq33}
    (V,H, \mathbf{W}, \mathbf{B}) \in X_{\delta_0}(0,\infty),
\end{align}
satisfying
\begin{equation}\label{eq34}
\|(V,H, \mathbf{W}, \mathbf{B})(t)\|_2^2
+\int_0^\infty
\left(
\|(V_\xi, \mathbf{W}_\xi, \mathbf{B}_\xi)(t)\|_2^2
+
\|H_\xi(t)\|_1^2
\right)\, \text{d}t
\le
C_0\,\|(V_0, H_0, \mathbf{W}_0, \mathbf{B}_0)\|_2^2.
\end{equation}
\end{theorem}
\section{Energy estimates}\label{sec4}
In this section, we assume that system \eqref{eq26}-\eqref{eq28} has solution locally in time, i.e., $\left(V,H,\mathbf{W}, \mathbf{B}\right)\in X_\delta(0,T)$ , for some $T>0$ and sufficiently small $\delta>0$, i.e., 
\begin{align}\label{eq35}
    \sup_{0\leq t\leq T}\lVert(V,H,\mathbf{W},\mathbf{B})(t)\rVert_{2}\leq \delta.
\end{align}
Then, by the Sobolev inequality, we have
\begin{align}\label{eq36}
    \sup_{0\leq t\leq T}\big\{\|(V, H, \mathbf{W}, \mathbf{B})(t)\|_{L^\infty}+\|(v-\bar{v}, h-\bar{h}, \mathbf{w}, v\mathbf{b})(t)\|_{L^\infty}\big\}\leq\delta.
\end{align}
Before proving Proposition \ref{proposition3.1}, let us first obtain some useful inequalities which will be used frequently later.
\begin{lemma}\label{lemma4.1}
Suppose that assumption \eqref{eq35} holds, then we have
\begin{equation}\label{eq37}
\begin{aligned}
    &|R_1|\leq C\left(|V_\xi||V_{\xi \xi}|+|V_\xi|^2\right),\\
    & |R_{1, \xi}|\leq C\left(V_{\xi \xi}^2+|V_\xi||V_{\xi \xi}|^2+|V_\xi||V_{\xi \xi \xi}|+|V_\xi|^2\right).
\end{aligned}
\end{equation}
\end{lemma}
\begin{proof}
    From $\eqref{eq27}_1$ we have 
    \begin{equation*}
        \begin{aligned}
            &R_1= \left(f'(v)- f'(\bar{v})\right)V_{\xi \xi} +\left(f'(v)-f'(\bar{v}) - f''(\bar{v})V_\xi\right)\bar{v}_\xi,\\[0.3cm]
            & R_{1, \xi} = \left(f'(v)-f'(\bar{v})\right)V_{\xi\xi\xi}+ \left(f'(v)-f'(\bar{v})-f''(\bar{v})V_{\xi}\right)\bar{v}_{\xi\xi} \\[0.3cm] &~~\qquad + \left[\left(f''(v)-f''(\bar{v})\right)V_{\xi\xi}+ \left(f''(v)-f''(\bar{v})\right)\bar{v}_{\xi}+ f''(\bar{v})V_{\xi\xi}\right]V_{\xi\xi} \\[0.3cm]&~~\qquad + \left[\left(f''(v)-f''(\bar{v})\right)V_{\xi\xi}+ \left(f''(v)-f''(\bar{v})-f'''(\bar{v})V_{\xi}\right)\bar{v}_{\xi}\right]\bar{v}_{\xi}. 
        \end{aligned}
    \end{equation*}
    Taking into account \eqref{eq36} and the boundedness of the travelling wave, i.e., $|\bar{v}_\xi|$, $|\bar{v}_{\xi \xi}|$, the proof Lemma follows immediately.  
\end{proof}
\begin{lemma}\label{lemma4.2}
Suppose that assumption \eqref{eq35} holds, then we have
\begin{equation}\label{eq38}
\begin{aligned}
    &|R_2|\leq C\left(|V_\xi|^2+|\mathbf{B}_\xi|^2\right),\\[0.2cm]
    & |R_{2, \xi}|\leq C\left(|V_{\xi \xi}||V_\xi|+ V_\xi^2+|\mathbf{B_\xi}|^2|V_{\xi \xi}|+|\mathbf{B}_\xi|^2+|\mathbf{B}_{\xi\xi}|^2\right),\\[0.2cm]
    & |R_{2, \xi \xi}|\leq C\left(|V_{\xi \xi \xi}||V_\xi|+ V_\xi^2+V_{\xi \xi}^2+|V_\xi||V_{\xi \xi}|^2+|\mathbf{B}_\xi|^2+|\mathbf{B}_{\xi \xi}|^2+|\mathbf{B}_{ \xi \xi}||\mathbf{B}_{ \xi \xi \xi}|+|\mathbf{B}_\xi|^2V_\xi+|\mathbf{B}_\xi|^2V_{\xi \xi}\right).
\end{aligned}
\end{equation}
\end{lemma}
\begin{proof}
    From $\eqref{eq27}_2$ we have 
\begin{equation*}
\begin{aligned}
R_2 &= -\left(p(v)-p(\bar{v})-p'(\bar{v})(v-\bar{v})\right)
      -\frac{|\mathbf{B}_\xi|^2}{2v^2},\\[0.3cm]
R_{2,\xi} &=
-\Big[(p'(v)-p'(\bar{v}))V_{\xi\xi}
+(p'(v)-p'(\bar{v})-p''(\bar{v})V_\xi)\bar{v}_\xi
+|\mathbf{B}_\xi|^2\left(\frac{1}{2v^2}-\frac{1}{2\bar{v}^2}\right)_\xi \\
&\qquad \qquad \qquad \qquad \qquad  +\frac{1}{2v^2}(|\mathbf{B}_\xi|^2)_\xi
+|\mathbf{B}_\xi|^2\left(\frac{1}{2\bar{v}^2}\right)_\xi
\Big].
\end{aligned}
\end{equation*}
For the sake of convenience define $h_1(v):= \frac{1}{2v^2}$ then  
\begin{equation*}
    \begin{aligned}
        R_{2,\xi} &= -\Big[(p'(v)-p'(\bar{v}))V_{\xi\xi}
+(p'(v)-p'(\bar{v})-p''(\bar{v})V_\xi)\bar{v}_\xi
+|\mathbf{B}_\xi|^2\left(h_1'(v)V_{\xi \xi}- \bar{v}_\xi(h_1'(v)-h_1'(\bar{v}))\right) \\
&\qquad \qquad \qquad \qquad \qquad \qquad\qquad +h_1(v)(2\mathbf{B}_\xi\,.\, \mathbf{B_{\xi \xi}})
+|\mathbf{B}_\xi|^2h_1(\bar{v})_\xi
\Big]
    \end{aligned}
\end{equation*}
\begin{equation*}
    \begin{aligned}
 R_{2, \xi \xi} = &-\bigg[(p'(v)-p'(\bar{v}))V_{\xi\xi\xi}
+\left(p'(v)-p'(\bar{v})-p''(\bar{v})V_{\xi}\right)\bar{v}_{\xi\xi} \\[0.2cm]
&+\left[\left(p''(v)-p''(\bar{v})\right)V_{\xi\xi}
+\left(p''(v)-p''(\bar{v})\right)\bar{v}_{\xi}
+p''(\bar{v})V_{\xi\xi}\right]V_{\xi\xi}\\[0.2cm]
&+\left[\left(p''(v)-p''(\bar{v})\right)V_{\xi\xi}
+\left(p''(v)-p''(\bar{v})-p'''(\bar{v})V_{\xi}\right)\bar{v}_{\xi}\right]\bar{v}_{\xi}\\[0.2cm]
&+\left(h_1'(v)V_{\xi \xi}- \bar{v}_\xi(h_1'(v)-h_1'(\bar{v}))\right)(2\mathbf{B}_\xi\,.\, \mathbf{B_{\xi \xi}})+h_1'(v)\left(V_{\xi \xi}+\bar{v}_\xi\right)\left(2\mathbf{B}_\xi\,.\, \mathbf{B}_{\xi \xi}\right)\\[0.2cm]
& + \left|\mathbf{B}_\xi\right|^2\left[h_1'(v)V_{\xi \xi \xi}+h_1'(v)V_{\xi \xi}^2+(h_1'(v)\bar{v}_\xi-h_1''(v)\bar{v}_\xi) V_{\xi \xi}-\bar{v}_{\xi \xi}\left(h_1'(v)-h_1'(\bar{v})\right)+\left(\bar{v}_\xi^2(h_1''(v)-h_1''(\bar{v}))\right)\right]\\[0.2cm]
& +2h_1(v)\left(|\mathbf{B}_{\xi \xi}|^2+\mathbf{B}_{\xi \xi}\, .\, \mathbf{B}_{\xi \xi \xi}\right)+ h_1(\bar{v})_{\xi \xi}|\mathbf{B}_{\xi \xi}|^2+ h_1(\bar{v})_\xi(2\mathbf{B}_\xi\,.\, \mathbf{B_{\xi \xi}})\bigg].
    \end{aligned}
\end{equation*}
By considering \eqref{eq36}, together with the boundedness of the travelling wave derivatives, namely
$|\bar{v}_\xi|$, $|\bar{v}_{\xi\xi}|$ and applying Young's inequality, the proof of  Lemma follows immediately.  
\end{proof}
\begin{lemma}\label{lemma4.3}
Suppose that assumption \eqref{eq35} holds, then we have
\begin{equation}\label{eq39}
\begin{aligned}
    &|R_3|\leq C\left(|V_\xi||\mathbf{W}_{\xi \xi}|+|V_\xi||\mathbf{B}_\xi|\right),\\
    & |R_{3, \xi}|\leq C\left(|V_\xi||\mathbf{W}_{\xi \xi \xi}|+|V_{\xi \xi}||W_{\xi \xi}|+|\bar{v}_\xi||V_{\xi}||\mathbf{W}_{\xi \xi}|+|V_\xi||\mathbf{B}_{\xi \xi}|+|V_{\xi \xi}||\mathbf{B}_{ \xi}|+|\bar{v}_\xi||V_{\xi}||\mathbf{B}_{\xi}|\right).
\end{aligned}
\end{equation}
\end{lemma}
\begin{proof}
    From $\eqref{eq27}_3$ we have 
    \begin{equation*}
        \begin{aligned}
            &R_3= \mu\left(\dfrac{1}{v}-\dfrac{1}{\bar{v}}\right)\mathbf{W}_{\xi \xi}+\left(\dfrac{1}{v}-\dfrac{1}{\bar{v}}\right)\mathbf{B}_{ \xi}.\\
        \end{aligned}
    \end{equation*}
    Define  $h_2(v) = \frac{1}{v}$ and differentiating $R_3$ once with respect to $\xi$ we have
    \begin{equation*}
        \begin{aligned}
            R_{3, \xi} = &\mu\left(h_2(v)-h_2(\bar{v})\right)\mathbf{W}_{\xi \xi \xi}+\mu\left[\left(h_2'(v)-h_2'(\bar{v})\right)V_{\xi \xi}+\left(h_2'(v)-h_2'(\bar{v})\bar{v}_\xi+h_2'(\bar{v}V_{\xi \xi})\right)\right]\mathbf{W}_{\xi \xi}\\[0.2cm]
            & +\left(h_2(v)-h_2(\bar{v})\right)\mathbf{B}_{\xi \xi}+\left[\left(h_2'(v)-h_2'(\bar{v})\right)V_{\xi \xi}+\left(h_2'(v)-h_2'(\bar{v})\bar{v}_\xi+h_2'(\bar{v}V_{\xi \xi})\right)\right]\mathbf{B}_{\xi}
        \end{aligned}
    \end{equation*}
    Taking into account \eqref{eq36} and the boundedness of the travelling wave leads to the proof of the Lemma. 
\end{proof}
\begin{lemma}\label{lemma4.4}
Suppose that assumption \eqref{eq35} holds, then we have
\begin{equation}\label{eq40}
\begin{aligned}
    &|R_4|\leq C\left(|V_\xi||\mathbf{B}_{\xi \xi}|+|V_{\xi \xi}||\mathbf{B}_{\xi}|+|V_\xi||\mathbf{B}_{\xi}|\right),\\
    & |R_{4, \xi}|\leq C\left(|V_{\xi \xi}||\mathbf{B}_{\xi \xi}|+|V_\xi||\mathbf{B}_{\xi \xi}|+|V_\xi||\mathbf{B}_{\xi \xi \xi}|+|\mathbf{B}_{\xi}||V_{\xi \xi \xi}|+|V_{\xi \xi}|^2|\mathbf{B}_\xi|+|V_{\xi \xi}||\mathbf{B}_{\xi}|+|V_\xi||\mathbf{B}_\xi|\right).
\end{aligned}
\end{equation}
\end{lemma}
\begin{proof}
    From $\eqref{eq27}_4$ we have 
    \begin{equation*}
        \begin{aligned}
            &R_4= \nu\left(\dfrac{1}{v^2}-\dfrac{1}{\bar{v}^2}\right)\mathbf{B}_{\xi \xi}+\dfrac{\nu}{2}\left(\dfrac{1}{v^2}-\dfrac{1}{\bar{v}^2}\right)_\xi\mathbf{B}_\xi,\\[0.3cm]
        \end{aligned}
    \end{equation*}
    Define  $h_3(v) := \frac{1}{v^2}$ and rearranging $R_4$ yields,
    \begin{equation*}
        \begin{aligned}
            R_4 = \nu \left(h_3(v)-h_3(\bar{v})\right)\mathbf{B}_{\xi \xi}+\dfrac{\nu}{2}\left(h_3'(v)V_{\xi \xi}+ \bar{v}_\xi(h_3'(v)-h_3'(\bar{v})) \right)\mathbf{B}_\xi
        \end{aligned}
    \end{equation*}
    Now differentiating once,
    \begin{equation*}
        \begin{aligned}
            R_{4, \xi}& = \nu\left(h_3'(v)V_{\xi \xi}+ \bar{v}_\xi(h_3'(v)-h_3'(\bar{v}))\right)\mathbf{B}_{\xi \xi}+ \nu \left(h_3(v)-h_3(\bar{v})\right)\mathbf{B}_{\xi \xi\xi}+\dfrac{\nu\mathbf{B}_{\xi \xi}}{2}\left(h_3'(v)V_{\xi \xi}+ \bar{v}_\xi(h_3'(v)-h_3'(\bar{v})) \right)\\[0.2cm]
            & + \frac{\nu\mathbf{B}_\xi}{2}\left[h_3'(v)V_{\xi \xi \xi}+h_3'(v)V_{\xi \xi}^2+(h_3'(v)\bar{v}_\xi-h_3''(v)\bar{v}_\xi) V_{\xi \xi}-\bar{v}_{\xi \xi}\left(h_3'(v)-h_3'(\bar{v})\right)+\left(\bar{v}_\xi^2(h_3''(v)-h_3''(\bar{v}))\right)\right]
        \end{aligned}
    \end{equation*}
By considering \eqref{eq36}, along with the boundedness of the travelling wave derivatives 
$|\bar{v}_\xi|$, $|\bar{v}_{\xi\xi}|$ and applying Young's inequality, the proof of Lemma follows immediately. 
\end{proof}
With the help of the above inequalities \eqref{eq36}-\eqref{eq40},  we are now going to prove Proposition \ref{proposition3.1} by the following series of Lemmas
\begin{lemma}\label{lemma4.5}
    Suppose that the assumptions of the proposition \ref{proposition3.1} hold, then we have the following estimate for $t\in [0, T]$ 
    \begin{equation}\label{eq41}
        \begin{aligned}
            &\lVert(V, H, \mathbf{W}, \mathbf{B})(t) \rVert^2- \sigma \int_0^t\int_{-\infty}^{\infty} \left(\dfrac{1}{p'(\bar{v})}\right)_\xi H^2 \text{d}\xi \text{d}\tau+ \sigma \int_0^t\int_{-\infty}^{\infty} \bar{v}_\xi\mathbf{W}^2\text{d}\xi \text{d}\tau\\[0.2cm]
            & + \int_0^t\lVert (V_\xi, \mathbf{W}_\xi, \mathbf{B}_\xi)(\tau)\rVert^2 \text{d}\tau \leq C\lVert (V_0, H_0, \mathbf{W}_0, \mathbf{B}_0)\rVert^2 + C\delta \int_0^t \lVert(V_{\xi \xi}, \mathbf{W}_{\xi \xi}, \mathbf{B}_{\xi \xi})(\tau)\rVert \text{d}\tau. 
        \end{aligned}
    \end{equation}
\end{lemma}
\begin{proof}
    Multiplying $\eqref{eq26}_1$, $\eqref{eq26}_2$, $\eqref{eq26}_3$ and $\eqref{eq26}_4$ by $V$, $-\frac{H}{p'(\bar{v})}$, $\bar{v}\mathbf{W}$ and $\mathbf{B}$, respectively, summing them up and integrating the resulting equality with respect to $\xi$. Then, doing integration by parts we obtain 
    \begin{equation}\label{eq42}
        \begin{aligned}
            &\dfrac{1}{2}\dfrac{\text{d}}{\text{d}t}\int \left(V^2-\dfrac{H^2}{p'(\bar{v})}  +\bar{v}|\mathbf{W}|^2+|\mathbf{B}|^2\right)\text{d}\xi  -\dfrac{\sigma}{2}\int \left(\dfrac{1}{p'(\bar{v})}\right)_\xi H^2\text{d}\xi \\[0.25cm]
            &+\dfrac{\sigma}{2}\int (\bar{v})_\xi \mathbf{W}^2\text{d}\xi + \int f'(\bar{v})V_\xi^2 \text{d}\xi + \mu \int \mathbf{W}_\xi^2\text{d}\xi   +\nu \int \dfrac{1}{\bar{v}^2}\mathbf{B}_\xi^2\text{d}\xi\\[0.25cm]
            &  = \int R_1 V \text{d}\xi- \int R_2\dfrac{H}{p'(\bar{v})}\text{d}\xi+ \int R_3 \bar{v}\mathbf{W}\text{d}\xi+\int R_4\mathbf{B}\text{d}\xi = \sum_{i=1}^4 I_i.
        \end{aligned}
    \end{equation}
    Now we estimate $I_i$, $i\in \{1,2,3,4\}$, with the help of Lemmas \ref{lemma4.1}-\ref{lemma4.4}, $\eqref{eq35}$ together with the boundedness of the travelling wave namely $|\bar{v}_\xi|$ and $|\bar{v}_{\xi \xi}|$, as follows
    \begin{equation}\label{eq43}
        \begin{aligned}
            |I_1+I_2| &\leq C\int \left(|V_\xi||V_{\xi \xi}|+|V_\xi|^2\right)|V|+(|V_\xi|^2+|\mathbf{B}_{ \xi}|^2)\left|\dfrac{H}{p'(\bar{v})}\right|\text{d}\xi \\[0.2cm]
            & \leq C\delta \int \left(|V_\xi|^2+|V_{\xi \xi}|^2+|\mathbf{B}_\xi|^2\right)\text{d}\xi.
        \end{aligned}
    \end{equation}
    Similarly,  
       \begin{equation}\label{eq44}
        \begin{aligned}
            |I_3+I_4| & \leq C \int\lVert\mathbf{W}\rVert_{L^\infty}(|V_\xi||\mathbf{W}_{\xi \xi}|+|V_\xi||\mathbf{B}_\xi|)   +\lVert H\rVert_{L^\infty}(|V_\xi||\mathbf{B}_{\xi \xi}|+|V_{\xi \xi}||\mathbf{B}_{\xi}|+|V_\xi||\mathbf{B}_{\xi}|)\text{d}\xi\\[0.2cm]
            & \leq C \lVert(H, \mathbf{W} )\rVert_{L^\infty}\int \left(|V_\xi|^2+|\mathbf{B}_\xi|^2+|V_{\xi \xi}|^2+|\mathbf{B}_{\xi \xi}|^2\right)\text{d}\xi\\[0.2cm]
            & \leq C\delta \int \left(|V_\xi|^2+|\mathbf{B}_\xi|^2+|V_{\xi \xi}|^2+|\mathbf{B}_{\xi \xi}|^2\right)\text{d}\xi. 
        \end{aligned}
    \end{equation}
    Substituting \eqref{eq43}-\eqref{eq44} in \eqref{eq42}, then choosing $\delta$ sufficiently small and integrating resulting inequality with respect to time $t$, we obtain \eqref{eq41}. Hence, Lemma \ref{lemma4.5} is proved.
\end{proof}
\begin{remark}
    We recall that in our case we consider the shock which is having speed $\sigma<0$ and from Lemma \ref{lemma2.2} we have $\bar{v}_\xi <0$ which implies $\left(\frac{1}{p'(\bar{v})}\right)_\xi= -\left(\frac{1}{p'(\bar{v})}\right)^2p''(\bar{v})\bar{v}_\xi>0$. Therefore the terms $\sigma \int_0^t\int_{-\infty}^{\infty} \bar{v}_\xi|\mathbf{W}|^2\text{d}\xi \text{d}\tau$ and  $- \sigma \int_0^t\int_{-\infty}^{\infty} \left(\frac{1}{p'(\bar{v})}\right)_\xi H^2 \text{d}\xi \text{d}\tau$ are supporting terms for us  in Lemma \ref{lemma4.5}.
\end{remark}
\begin{lemma}\label{lemma4.6}
    Suppose that the assumptions of the proposition \ref{proposition3.1} hold, then we have the following estimate for $t\in [0, T]$ as
    \begin{equation}\label{eq45}
        \begin{aligned}
            &\lVert(V, H, \mathbf{W}, \mathbf{B})(t) \rVert_1^2+\int_0^t\lVert (V_\xi, \mathbf{W}_\xi, \mathbf{B}_\xi)(\tau)\rVert_1^2  \text{d}\tau
              \leq C\lVert (V_0, H_0, \mathbf{W}_0, \mathbf{B}_0)\rVert_1^2 . 
        \end{aligned}
    \end{equation}
\end{lemma}
\begin{proof}
    Multiplying $\eqref{eq26}_1$, $\eqref{eq26}_2$, $\eqref{eq26}_3$ and $\eqref{eq26}_4$ by $-V_{\xi \xi}$, $\frac{H_{\xi \xi}}{p'(\bar{v})}$, $-\bar{v}\mathbf{W}_{\xi \xi}$ and $-\mathbf{B}_{\xi \xi}$, respectively.  Summing them up and integrating the resulting equality with respect to $\xi$. Then, by performing integration by parts, we obtain
    \begin{equation}\label{eq46}
        \begin{aligned}
            &\dfrac{1}{2}\dfrac{\text{d}}{\text{d}t}\int \left(V_\xi^2-\dfrac{H_\xi^2}{p'(\bar{v})}+\bar{v}|\mathbf{W}_\xi|^2+|\mathbf{B}_\xi|^2\right)\text{d}\xi - \dfrac{\sigma}{2}\int \left(\frac{1}{p'(\bar{v})}\right)_\xi H_\xi^2\text{d}\xi\\[0.2cm]
            & +\int f'(\bar{v}) V_{\xi \xi}^2\text{d}\xi\, +\, \mu\int |\mathbf{W}_{\xi \xi}|^2\text{d}\xi\, +\,\int \left(\dfrac{\nu}{\bar{v}^2}\right)|\mathbf{B}_{\xi \xi}|^2\text{d}\xi\\[0.2cm]
            &= \int R_2 \dfrac{H_{\xi \xi}}{p'(\bar{v})}\text{d}\xi-\int R_1V_{\xi \xi}\, \text{d}\xi-\int f''(\bar{v})\bar{v}_\xi V_{\xi \xi}V_\xi \text{d}\xi-\int R_3\bar{v}\mathbf{W}_{\xi \xi}\text{d}\xi- \int R_4\mathbf{B_{\xi \xi}}\text{d}\xi\\[0.2cm]
            & +\int \left(\dfrac{1}{p'(\bar{v})}\right)_\xi H_\xi H_t \text{d}\xi - \int \left(\bar{v}\right)_\xi\mathbf{W}_\xi\,.\,\mathbf{W}_t\text{d}\xi- \int \left(\dfrac{\nu}{2\bar{v}^2}\right)\mathbf{B}_{\xi \xi}\,.\, \mathbf{B}_\xi\text{d}\xi+\dfrac{\sigma}{2}\int (\bar{v})_\xi|\mathbf{W}_\xi|^2\text{d}\xi\\[0.2cm]
            &=: \sum_{i=5}^{13} I_i.
        \end{aligned}
    \end{equation}
   By exploiting Lemmas \ref{lemma4.1}-\ref{lemma4.4}, \eqref{eq35} along with the boundedness of the travelling wave, we estimate $I_i ~(i=5,...,13)$ as follows
   \begin{equation}\label{eq47}
       \begin{aligned}
           |I_6+I_7| &\leq C\int \left(\left(|V_\xi||V_{\xi \xi}|+|V_\xi|^2\right)|V_{\xi \xi}|+|V_\xi||V_{\xi \xi}|\right)\text{d}\xi\\[0.2cm]
           &\leq (\epsilon+C \delta)\int |V_{\xi \xi}|^2\text{d}\xi+ C \int|V_\xi|^2\text{d}\xi.
       \end{aligned}
   \end{equation}
   Note that,
   \begin{equation}\label{eq48}
       \begin{aligned}
           I_5+I_{10} = \int\left(R_2\dfrac{H_\xi}{p'(\bar{v})}\right)_\xi \text{d}\xi- \int R_{2,\xi}\dfrac{H_\xi}{p'(\bar{v})}\text{d}\xi +\int \left(\dfrac{1}{p'(\bar{v})}\right)_\xi p'(\bar{v})V_\xi H_\xi \text{d}\xi
       \end{aligned}
   \end{equation}
   and
   \begin{equation}\label{eq49}
       \begin{aligned}
          &\int \left(\dfrac{1}{p'(\bar{v})}\right)_\xi p'(\bar{v})V_\xi H_\xi \text{d}\xi \leq -\dfrac{\sigma}{4}\int\left(\dfrac{1}{p'(\bar{v})}\right)_\xi H_\xi^2\text{d}\xi +C\int|V_\xi|^2\text{d}\xi,\\[0.2cm]
          &\left|\int R_{2,\xi}\dfrac{H_\xi}{p'(\bar{v})}\text{d}\xi\right| \leq C\int\left(|V_{\xi \xi}||V_\xi|+ V_\xi^2+|\mathbf{B_\xi}|^2|V_{\xi \xi}|+|\mathbf{B}_\xi|^2+|\mathbf{B}_{\xi\xi}|^2\right)|H_\xi|\text{d}\xi\\
          &\qquad \qquad \qquad \quad \leq C\delta \int\left(|V_\xi|^2+|V_{\xi \xi}|^2+|\mathbf{B}_\xi|^2+|\mathbf{B}_{\xi \xi}|^2\right)\text{d}\xi.
       \end{aligned}
   \end{equation}
   \vspace{0.1cm}

From $\eqref{eq26}_2$ 
\begin{equation}\label{eq51}
    \begin{aligned}
        |I_{11}|&= \left|\int \left(\bar{v}\right)_\xi\mathbf{W}_\xi\,.\,\left(R_3+\sigma\mathbf{W}_\xi+\frac{\mathbf{B}_\xi}{\bar{v}}+\mu \frac{\mathbf{W}_{\xi \xi}}{\bar{v}}\right )\text{d}\xi\right|\\[0.2cm]
        &\leq C \int |R_3||\mathbf{W}_\xi|\text{d}\xi+C\int |\mathbf{W}_\xi|^2\text{d}\xi+C\int |\mathbf{W}_\xi||\mathbf{B}_\xi|\text{d}\xi+C\int |\mathbf{W}_\xi||\mathbf{W}_{\xi \xi}|\text{d}\xi\\[0.2cm]
        & \leq C\int \left(|V_\xi||\mathbf{W}_{\xi \xi}|+|V_\xi||\mathbf{B}_\xi|\right)|\mathbf{W}_{\xi}|\text{d}\xi+(C+\epsilon)\int |\mathbf{W}_\xi|^2\text{d}\xi + C_\epsilon\int \left(|\mathbf{B}_\xi|^2+|\mathbf{W}_{\xi \xi}|^2\right)\text{d}\xi\\[0.2cm]
        &\leq C\delta\int |V_\xi|^2\text{d}\xi+C(\delta+C_\epsilon) \int\left(|\mathbf{B}_\xi|^2+|\mathbf{W}_{\xi \xi}|^2\right)\text{d}\xi+(C+\epsilon)\int |\mathbf{W}_\xi|^2\text{d}\xi.
    \end{aligned}
\end{equation}
Meanwhile, one readily obtains
\begin{equation}\label{eq52}
    \begin{aligned}
        |I_{12}+I_{13}|\leq C\int \left(|\mathbf{W_\xi}|^2+|\mathbf{B}_{\xi}|^2+|\mathbf{B}_{ \xi \xi}|^2\right)\text{d}\xi,
    \end{aligned}
\end{equation}
\vspace{0.2cm}
\begin{equation}\label{eq50}
    \begin{aligned}
        |I_8|\leq C\int \left(|V_\xi||\mathbf{W}_{\xi \xi}|+|V_\xi||\mathbf{B}_\xi|\right)|\mathbf{W}_{\xi \xi}|\text{d}\xi \leq C\delta\int(|\mathbf{B}_\xi|^2+|\mathbf{W}_{\xi \xi}|^2)\text{d}\xi.
    \end{aligned}
\end{equation}
Similarly,
\begin{equation}\label{eq53}
    \begin{aligned}
        |I_9|&\leq C\int \left(|V_\xi||\mathbf{B}_{\xi \xi}|+|V_{\xi \xi}||\mathbf{B}_{\xi}|+|V_\xi||\mathbf{B}_{\xi}|\right)|\mathbf{B}_{\xi \xi}|\text{d}\xi\\[0.2cm]
        &\leq C\delta \int \left(|\mathbf{B}_\xi|^2+|V_{\xi \xi}|^2+|\mathbf{B}_{\xi \xi}|^2\right)\text{d}\xi.
    \end{aligned}
\end{equation}
Substituting \eqref{eq47}-\eqref{eq53} in \eqref{eq46} and integrating the resulting inequality with respect to time $t$, we obtain
 \begin{equation}\label{eq54}
        \begin{aligned}
            &\dfrac{1}{2}\int \left(V_\xi^2-\dfrac{H_\xi^2}{p'(\bar{v})}+\bar{v}|\mathbf{W}_\xi|^2+|\mathbf{B}_\xi|^2\right)\text{d}\xi - \dfrac{\sigma}{4}\int_0^t\int \left(\frac{1}{p'(\bar{v})}\right)_\xi H_\xi^2\text{d}\xi\text{d}\tau \\[0.2cm]
            & +\int_0^t\int f'(\bar{v}) V_{\xi \xi}^2\text{d}\xi\text{d}\tau\, +\, \mu\int_0^t\int |\mathbf{W}_{\xi \xi}|^2\text{d}\xi\text{d}\tau\, +\,\int_0^t\int \left(\dfrac{\nu}{\bar{v}^2}\right)|\mathbf{B}_{\xi \xi}|^2\text{d}\xi\text{d}\tau\\[0.2cm]
            &\leq C \lVert\left(V_{0, \xi}, H_{0,\xi}, \mathbf{W}_{0,\xi}, \mathbf{B}_{0,\xi}\right)\rVert^2+(\epsilon+C\delta)\int_0^t\int |V_{\xi \xi}|^2\text{d}\xi\text{d}\tau+C(\delta+2)\int_0^t\int |\mathbf{B}_{\xi \xi}|^2\text{d}\xi\text{d}\tau \\[0.2cm]
            &+C(\delta+C_\epsilon)\int_0^t\int |\mathbf{W}_{\xi \xi}|^2\text{d}\xi\text{d}\tau +C(\delta+1)\int_0^t\int |V_\xi|^2 \text{d}\xi\text{d}\tau + (C+\epsilon)\int_0^t\int|\mathbf{W}_\xi|^2 \text{d}\xi\text{d}\tau\\[0.2cm]           &+C(1+\delta+C_\epsilon)\int_0^t\int|\mathbf{B}_\xi|^2\text{d}\xi\text{d}\tau.\\
        \end{aligned}
    \end{equation}
    Now using Lemma \ref{lemma4.5}, the right-hand side of the inequality \eqref{eq54} becomes
    \begin{equation}\label{eq55}
        \begin{aligned}
            & \leq C \lVert\left(V_{0}, H_{0}, \mathbf{W}_{0}, \mathbf{B}_{0}\right)\rVert_1^2+(\epsilon+C\delta)\int_0^t\int |V_{\xi \xi}|^2\text{d}\xi\text{d}\tau+C(\delta+2)\int_0^t\int |\mathbf{B}_{\xi \xi}|^2\text{d}\xi\text{d}\tau \\[0.2cm]
            & +C(\delta+C_\epsilon)\int_0^t\int |\mathbf{W}_{\xi \xi}|^2\text{d}\xi\text{d}\tau+C(\delta+1)\delta\int_0^t\int |V_{\xi \xi}|^2 \text{d}\xi\text{d}\tau + (C+\epsilon)\delta\int_0^t\int|\mathbf{W}_{\xi \xi}|^2 \text{d}\xi\text{d}\tau\\[0.2cm]            &+C(1+\delta+C_\epsilon)\delta\int_0^t\int|\mathbf{B}_{\xi \xi}|^2\text{d}\xi\text{d}\tau.\\
        \end{aligned}
    \end{equation}
    Choosing suitably small $\epsilon$ and sufficiently small $\delta$, we have
    \begin{equation}\label{eq56}
        \begin{aligned}
           &\int \left(V_\xi^2+H_\xi^2+  |\mathbf{W}_\xi|^2+|\mathbf{B}_\xi|^2\right)\text{d}\xi + \int_0^t\int V_{\xi \xi}^2+ |\mathbf{W}_{\xi \xi}|^2 +|\mathbf{B}_{\xi \xi}|^2 \text{d}\xi\text{d}\tau\\[0.25cm] 
           & \leq C \lVert\left(V_{0}, H_{0}, \mathbf{W}_{0}, \mathbf{B}_{0}\right)\rVert_1^2
        \end{aligned}
    \end{equation}
    and along with Lemma \ref{lemma4.5}, \eqref{eq45} follows immediately. Thus, Lemma \ref{lemma4.6} is proved.
\end{proof}
As we may not have, in general, the positive lower bound for $|\bar{v}_\xi|$ hence it is necessary to deal with $\int_0^t\lVert H_\xi(\tau)\rVert^2\text{d}\tau$ separately. For this we have
\begin{lemma}\label{lemma4.7}
    Suppose that the assumptions of the proposition \ref{proposition3.1} hold, then we have the following estimate for $t\in [0, T]$ as
    \begin{equation}\label{eq57}
        \begin{aligned}
            \int_0^t \lVert H_\xi(\tau)\rVert^2 \text{d}\tau \leq C\lVert (V_0, H_0, \mathbf{W}_0, \mathbf{B}_0)\rVert_1^2 .
        \end{aligned}
    \end{equation}
\end{lemma}
\begin{proof}
    Multiplying $\eqref{eq26}_1$ by $H_\xi$ and using $\eqref{eq26}_2$ we have
    \begin{equation}\label{eq58}
        \begin{aligned}
            H_\xi^2 = &\left(VH_\xi\right)_t - V\left[R_{2, \xi}+\sigma H_{\xi \xi}-(p'(\bar{v})\bar{v}_\xi)_\xi\right]\\[0.25cm] - &\sigma V_\xi H_\xi- f'(\bar{v})V_{\xi \xi}H_\xi-f''(\bar{v})\bar{v}_\xi V_\xi H_\xi-R_1H_\xi. 
        \end{aligned}
    \end{equation}
    Integrating \eqref{eq58} with respect to $t$, $\xi$ and performing integration by parts  along with Lemma \ref{lemma4.6} leads to
    \begin{equation}\label{eq59}
        \begin{aligned}
            \int_0^t\int H_\xi^2\text{d}\xi &\text{d}\tau = \int VH_\xi\text{d}\xi-\int V_0H_{0,\xi}\text{d}\xi - \int_0^t\int V\left[R_{2, \xi}-(p'(\bar{v})V_\xi)_\xi\right]\text{d}\xi \text{d}\tau\\[0.25cm]  &- \int_0^t\int \left[f'(\bar{v})V_{\xi \xi}H_\xi+f''(\bar{v})\bar{v}_\xi V_\xi H_\xi+R_1H_\xi\right]\text{d}\xi \text{d}\tau\\[0.25cm]
            & \leq C\lVert (V_0, H_0, \mathbf{W}_0, \mathbf{B}_0)\rVert_1^2+ C\int_0^t\int \left(|V_\xi|^2+|\mathbf{B}_\xi|^2\right)|V_\xi|\text{d}\xi \text{d}\tau- C\int_0^t\int|V_\xi|^2\text{d}\xi \text{d}\tau\\[0.25cm]
            & +\int_0^t\int \dfrac{1}{2}H_\xi^2\text{d}\xi \text{d}\tau+C\int_0^t\int \left(|V_\xi|^2+|V_{\xi \xi}|^2+|R_1|^2\right)\text{d}\xi \text{d}\tau\\[0.25cm]
             & \leq C\lVert (V_0, H_0, \mathbf{W}_0, \mathbf{B}_0)\rVert_1^2+\int_0^t\int \dfrac{1}{2}H_\xi^2\text{d}\xi \text{d}\tau +C\int_0^t\int \left(|V_\xi|^2+|V_{\xi \xi}|^2+|\mathbf{B}_\xi|^2\right)\text{d}\xi \text{d}\tau\\[0.25cm]
             & \leq C\lVert (V_0, H_0, \mathbf{W}_0, \mathbf{B}_0)\rVert_1^2+\int_0^t\int \dfrac{1}{2}H_\xi^2\text{d}\xi \text{d}\tau.
        \end{aligned}
    \end{equation}
    From \eqref{eq59}, the Lemma \ref{lemma4.7} follows directly.
\end{proof}
\begin{lemma}\label{lemma4.8}
    Suppose that the assumptions of the proposition \ref{proposition3.1} hold, then we have the following estimate for $t\in [0, T]$ as
    \begin{equation}\label{eq60}
        \begin{aligned}
            \lVert(V_{\xi \xi}, H_{\xi \xi}, \mathbf{W}_{\xi \xi}, \mathbf{B}_{\xi \xi})(t)\rVert^2+ \int_0^t \lVert(V_{\xi \xi \xi}, \mathbf{W}_{\xi \xi \xi}, \mathbf{B}_{\xi \xi \xi})(\tau)\rVert^2\text{d}\tau \leq C\lVert (V_0, H_0, \mathbf{W}_0, \mathbf{B}_0)\rVert_2^2 .
        \end{aligned}
    \end{equation}
\end{lemma}
\begin{proof}
   Differentiate $\eqref{eq26}_1$ twice with respect to $\xi$ and multiply the resulting equation by $V_{\xi\xi}$. 
For $\eqref{eq26}_2$, first multiply by $-\frac{1}{p'(\bar v)}$ then differentiate twice with respect to $\xi$ and multiply the resulting equation by $H_{\xi\xi}$. 
For $\eqref{eq26}_3$, first multiply by $\bar v$ then differentiate twice with respect to $\xi$ and multiply the resulting equation by $\mathbf{W}_{\xi\xi}$. 
Finally, differentiate $\eqref{eq26}_4$ twice with respect to $\xi$ and multiply by $\mathbf{B}_{\xi\xi}$. 
Summing these equations and integrating with respect to $\xi$, we obtain, after integration by parts,
 \begin{equation}\label{eq61}
        \begin{aligned}
            &\dfrac{1}{2}\dfrac{\text{d}}{\text{d}t}\int \left(V_{\xi \xi}^2-\dfrac{H_{\xi \xi}^2}{p'(\bar{v})}+\bar{v}|\mathbf{W}_{\xi \xi}|^2+|\mathbf{B}_{\xi \xi}|^2\right)\text{d}\xi - \dfrac{\sigma}{2}\int \left(\frac{1}{p'(\bar{v})}\right)_\xi H_{\xi \xi}^2\text{d}\xi \\[0.2cm]
            & +\int f'(\bar{v}) V_{\xi \xi \xi}^2\text{d}\xi\, +\, \mu\int |\mathbf{W}_{\xi \xi \xi}|^2\text{d}\xi\, +\,\int \left(\dfrac{\nu}{\bar{v}^2}\right)|\mathbf{B}_{\xi \xi \xi}|^2\text{d}\xi\text{d}\tau\\[0.2cm]
            &=  \int R_{1,\xi \xi}V_{\xi \xi} \text{d}\xi+\int (\bar{v}R_3)_{\xi \xi}\mathbf{W}_{\xi \xi} \text{d}\xi + \int  R_{4,\xi \xi}\mathbf{B}_{\xi \xi} \text{d}\xi -\int R_{2, \xi \xi}\left(\dfrac{1}{p'(\bar{v})}\right)H_{\xi \xi}\text{d}\xi \\[0.2cm]
            & -2\int \left(\dfrac{1}{p'(\bar{v})}\right)_\xi\left(p'(\bar{v})V_\xi\right)_\xi H_{\xi \xi}\text{d}\xi-\int \left(\dfrac{1}{p'(\bar{v})}\right)_{\xi \xi}p'(\bar{v})V_\xi H_{\xi \xi}\text{d}\xi-\int \dfrac{\bar{v}_{\xi \xi}}{\bar{v}}\mathbf{W}_{\xi \xi}\, .\, \mathbf{B}_\xi\text{d}\xi\\[0.21cm]
            & -\mu\int \dfrac{\bar{v}_{\xi \xi}}{\bar{v}}\left|\mathbf{W}_{\xi \xi}\right|^2\text{d}\xi-2 \int \bar{v}_\xi \mathbf{W}_{\xi \xi}\, .\, \left(\dfrac{\mathbf{B}_\xi}{\bar{v}}\right)_\xi \text{d}\xi- 2 \int \bar{v}_\xi \mathbf{W}_{\xi \xi}\, .\, \left(\dfrac{\mathbf{W}_{\xi \xi}}{\bar{v}}\right)_\xi \text{d}\xi\\[0.21cm]
            & +\sigma \int \bar{v}\mathbf{W}_{\xi \xi \xi}\, .\, \mathbf{W}_{\xi \xi}\text{d}\xi+ \int f'(\bar{v})_{\xi \xi \xi}V_\xi V_{\xi \xi}\text{d}\xi+ 2\int f'(\bar{v})_{\xi \xi}V_{\xi \xi}^2\text{d}\xi-\int \left(\dfrac{\nu}{\bar{v}^2}\right)_\xi \mathbf{B}_{\xi \xi}\, . \,\mathbf{B}_{\xi \xi \xi}\text{d}\xi\\[0.21cm]
            & -\int \left(\left(\dfrac{\nu}{\bar{v}^2}\right)_\xi \mathbf{B}_\xi\right)_{\xi \xi}\, .\, \mathbf{B}_{\xi \xi}\text{d}\xi =:\sum_{i=14}^{28}I_i.
        \end{aligned}
    \end{equation}
    Using Lemmas \ref{lemma4.1}-\ref{lemma4.4}, \ref{lemma4.6}-\ref{lemma4.7} and applying Cauchy-Schwartz's inequality together with the boundedness of the viscous shock, we estimate terms $I_i (i = 14,...,28)$ as follows
    \begin{equation}\label{eq62}
        \begin{aligned}
            I_{14} &= -\int R_{1, \xi}V_{\xi \xi \xi}\text{d}\xi\\[0.21cm]
            & \leq C \int \left(V_{\xi \xi}^2+|V_\xi||V_{\xi \xi}|^2+|V_\xi||V_{\xi \xi \xi}|+|V_\xi|^2\right)|V_{\xi \xi \xi}|\text{d}\xi\\[0.21cm]
            & \leq C \lVert V_{\xi \xi}\rVert_{L^\infty}\lVert V_{\xi \xi}\rVert \lVert V_{\xi \xi \xi}\rVert+ C \delta\lVert V_{\xi \xi}\rVert_{L^\infty}\lVert V_{\xi \xi}\rVert \lVert V_{\xi \xi \xi}\rVert + C\delta \lVert V_{\xi \xi \xi}\rVert^2 + C\delta \left(\lVert V_{\xi }\rVert^2+\lVert V_{\xi \xi \xi}\rVert^2\right) \\[0.21cm]
            & \leq C\delta \left(\lVert V_{\xi \xi}\rVert + \lVert V_{\xi \xi \xi}\rVert\right)\lVert V_{\xi \xi \xi}\rVert +  C\delta \left( \lVert V_{\xi}\rVert^2+\lVert V_{\xi \xi \xi}\rVert^2 \right)\\[0.21cm]
            &\leq  C\delta \left(\lVert V_{\xi}\rVert^2+\lVert V_{\xi \xi }\rVert^2+\lVert V_{\xi \xi \xi}\rVert^2 \right) 
        \end{aligned}
        \end{equation}
        
        \begin{equation}\label{eq63}
            \begin{aligned}
                I_{18} =& -2\int \left(p'(\bar{v})V_{\xi \xi}+p''(\bar{v})\bar{v}_\xi V_\xi\right) \left(\dfrac{1}{p'(\bar{v})}\right)_\xi H_{\xi \xi} \text{d}\xi\\[0.21cm]
                &\leq  -\dfrac{\sigma}{4}\int \left(\frac{1}{p'(\bar{v})}\right)_\xi H_{\xi \xi}^2\text{d}\xi +C \int V_{\xi \xi}^2\text{d}\xi +C \int V_\xi^2 \text{d}\xi
            \end{aligned}
        \end{equation}
        \vspace{0.5cm}
        \begin{equation}\label{eq64}
            \begin{aligned}
                I_{19} &= \int \left[\left(\dfrac{1}{p'(\bar{v})}\right)_{\xi \xi}p'(\bar{v})V_\xi \right]_\xi H_{\xi}\text{d}\xi\\[0.21cm]
                & \leq C \left(\int V_{\xi \xi}^2\text{d}\xi+ \int V_{\xi}^2\text{d}\xi + \int H_{ \xi}^2\text{d}\xi \right).
            \end{aligned}
        \end{equation}
        Meanwhile, it is straightforward to deduce that
        \begin{equation}\label{eq65}
            \begin{aligned}
                |I_{25}+I_{26}| \leq (\epsilon+C_\epsilon)\int |V_{\xi \xi}|^2\text{d}\xi+ C_\epsilon\int  |V_\xi|^2\text{d}\xi, 
            \end{aligned}
        \end{equation}
       
        \begin{equation}\label{eq66}
            \begin{aligned}
                |I_{20}+I_{21}| \leq \epsilon \lVert \mathbf{B}_\xi\rVert^2+ C_\epsilon\lVert \mathbf{W}_{ \xi \xi}\rVert^2
            \end{aligned}
        \end{equation}
     and 
      \begin{equation}\label{eq67}
            \begin{aligned}
                |I_{22}+I_{23}+I_{24}| \leq C\left( \lVert \mathbf{B}_\xi\rVert^2+ \lVert \mathbf{W}_{ \xi \xi}\rVert^2+ \lVert \mathbf{B}_{\xi \xi}\rVert^2+ \lVert \mathbf{W}_{ \xi \xi \xi}\rVert^2\right).
            \end{aligned}
        \end{equation}
        Now for $I_{17}$,
        \begin{equation}\label{eq68}
\begin{aligned}
|I_{17}| &\leq C \int \biggr( |V_{\xi\xi\xi}|\,|V_\xi| + V_\xi^2 + V_{\xi\xi}^2 
+ |V_\xi|\,|V_{\xi\xi}|^2 + |\mathbf{B}_\xi|^2 + |\mathbf{B}_{\xi\xi}|^2+ |\mathbf{B}_\xi||\mathbf{B}_{\xi\xi}|^2+ |\mathbf{B}_\xi||V_{\xi\xi}|^2  \\
&\qquad \qquad \qquad \qquad +|\mathbf{B}_\xi|^2|V_{\xi \xi}|^2 +|\mathbf{B}_\xi|^2|V_{\xi \xi \xi}|+  |\mathbf{B}_{\xi\xi}||\mathbf{B}_{\xi\xi\xi}| + |\mathbf{B}_\xi|^2 V_\xi 
+ |\mathbf{B}_\xi|^2 V_{\xi\xi} \biggr) |H_{\xi\xi}|\, d\xi \\[0.21cm]
& \leq C \lVert V_\xi \rVert_{L^\infty} \lVert V_{\xi \xi \xi}\rVert \lVert H_{\xi \xi}\rVert+ C\lVert V_\xi \rVert_{L^\infty} \lVert V_{ \xi}\rVert \lVert H_{\xi \xi}\rVert + C\lVert V_{\xi \xi} \rVert_{L^\infty} \lVert V_{\xi \xi}\rVert \lVert H_{\xi \xi}\rVert\\[0.21cm]
& + C\delta \lVert V_{\xi \xi} \rVert_{L^\infty} \lVert V_{\xi \xi}\rVert \lVert H_{\xi \xi}\rVert+C \lVert \mathbf{B}_{\xi} \rVert_{L^\infty} \lVert \mathbf{B}_{\xi}\rVert \lVert H_{\xi \xi}\rVert + C \lVert \mathbf{B}_{\xi \xi} \rVert_{L^\infty} \lVert \mathbf{B}_{\xi \xi} \rVert \lVert H_{\xi \xi}\rVert\\[0.21cm]
& + C\delta \lVert \mathbf{B}_{\xi \xi} \rVert_{L^\infty} \lVert \mathbf{B}_{\xi \xi} \rVert \lVert H_{\xi \xi}\rVert+ C\delta \lVert V_{\xi \xi} \rVert_{L^\infty} \lVert V_{\xi \xi} \rVert \lVert H_{\xi \xi}\rVert + C\delta \lVert \mathbf{B}_{ \xi} \rVert_{L^\infty} \lVert V_{\xi  \xi \xi} \rVert \lVert H_{\xi \xi}\rVert\\[0.21cm]
& + \lVert \mathbf{B}_{\xi \xi} \rVert_{L^\infty}\lVert \mathbf{B}_{\xi \xi \xi} \rVert\lVert H_{\xi \xi}\rVert+C\delta \lVert \mathbf{B}_{\xi} \rVert_{L^\infty} \lVert \mathbf{B}_{\xi} \rVert\lVert H_{\xi \xi}\rVert+ C\delta \lVert \mathbf{B}_{\xi} \rVert_{L^\infty} \lVert V_{\xi \xi} \rVert\lVert H_{\xi \xi}\rVert\\[0.5cm]
& \leq C\delta \left(\lVert V_\xi \rVert+ \lVert V_{\xi \xi} \rVert\right) \left(\lVert V_{\xi \xi \xi}\rVert +\lVert V_{\xi }\rVert \right) + C\delta \left(\lVert V_{\xi \xi}\rVert+\lVert V_{\xi \xi \xi}\rVert\right) \lVert V_{\xi \xi}\rVert \\[0.21cm]&+ C \delta \left(\lVert \mathbf{B}_\xi\rVert+\lVert \mathbf{B}_{\xi \xi}\rVert\right)\lVert \mathbf{B}_{\xi}\rVert + C \delta \left(\lVert \mathbf{B}_{\xi \xi}\rVert+\lVert \mathbf{B}_{\xi \xi  \xi}\rVert\right)\lVert \mathbf{B}_{\xi \xi}\rVert + C\delta \left(\lVert V_{\xi \xi}\rVert+ \lVert V_{\xi \xi\xi}\rVert\right)\lVert V_{\xi \xi}\rVert\\[0.21cm]
& +C \delta \left(\lVert \mathbf{B}_\xi\rVert+\lVert \mathbf{B}_{\xi \xi}\rVert\right)\lVert V_{\xi \xi \xi}\rVert + C \delta \left(\lVert \mathbf{B}_{\xi \xi}\rVert+\lVert \mathbf{B}_{\xi \xi  \xi}\rVert\right)\lVert \mathbf{B}_{\xi \xi \xi}\rVert+C \delta \left(\lVert \mathbf{B}_\xi\rVert+\lVert \mathbf{B}_{\xi \xi}\rVert\right)\lVert V_{\xi \xi}\rVert\\[0.5cm]
& \leq C\delta \left(\lVert V_{\xi}\rVert^2+ \lVert V_{\xi \xi}\rVert^2+ \lVert V_{\xi \xi \xi}\rVert^2 + \lVert \mathbf{B}_{\xi}\rVert^2+\lVert \mathbf{B}_{\xi \xi}\rVert^2 + \lVert \mathbf{B}_{\xi \xi \xi}\rVert^2\right).
\end{aligned}
\end{equation}
Similarly for $I_{16}$,
\begin{equation}\label{eq69}
    \begin{aligned}
        I_{16} &= -\int R_{4, \xi}\mathbf{B}_{\xi \xi \xi}\text{d}\xi\\[0.21cm]
        & \leq C \int \left(|V_{\xi \xi}||\mathbf{B}_{\xi \xi}|+|V_\xi||\mathbf{B}_{\xi \xi}|+|V_\xi||\mathbf{B}_{\xi \xi \xi}|+|\mathbf{B}_{\xi}||V_{\xi \xi \xi}|+|V_{\xi \xi}|^2|\mathbf{B}_\xi|+|V_{\xi \xi}||\mathbf{B}_{\xi}|+|V_\xi||\mathbf{B}_\xi| \right)|\mathbf{B}_{\xi \xi \xi}|\text{d}\xi\\[0.25cm]
        & \leq C \lVert \mathbf{B}_{\xi \xi}\rVert_{L^\infty}\lVert V_{\xi \xi}\rVert \lVert \mathbf{B}_{\xi \xi \xi}\rVert +C \delta \lVert \mathbf{B}_{\xi \xi}\rVert\lVert \mathbf{B}_{\xi \xi \xi}\rVert +C\delta \lVert \mathbf{B}_{\xi \xi \xi}\rVert^2+C\delta \lVert V_{\xi \xi \xi}\rVert\lVert \mathbf{B}_{\xi \xi \xi}\rVert\\[0.21cm]
        &\quad+C\delta  \lVert V_{\xi \xi}\rVert_{L^\infty}\lVert V_{\xi \xi}\rVert \lVert \mathbf{B}_{\xi \xi \xi}\rVert+C\delta \lVert V_{ \xi \xi}\rVert\lVert \mathbf{B}_{\xi \xi \xi}\rVert+C\delta \lVert V_{ \xi}\rVert\lVert \mathbf{B}_{\xi \xi \xi}\rVert\\[0.4cm]
        & \leq C\delta \left(\lVert \mathbf{B}_{ \xi \xi}\rVert^2+ \lVert \mathbf{B}_{\xi \xi \xi}\rVert^2 + \lVert V_{ \xi \xi}\rVert^2 + \lVert V_{ \xi \xi \xi}\rVert^2 +\lVert V_{ \xi}\rVert^2\right)
    \end{aligned}
\end{equation}
and for $I_{15}$, 
\begin{equation}\label{eq70}
    \begin{aligned}
        I_{15} &= - \int \left(\bar{v}R_{3, \xi}+ \bar{v}_\xi R_{3}\right) \mathbf{W}_{\xi \xi \xi} \text{d}\xi\\[0.21cm]
        & \leq C \int \left(|V_\xi||\mathbf{W}_{\xi \xi \xi}|+|V_{\xi \xi}||W_{\xi \xi}|+|V_{\xi}||\mathbf{W}_{\xi \xi}|+|V_\xi||\mathbf{B}_{\xi \xi}|+|V_{\xi \xi}||\mathbf{B}_{ \xi}|+|V_{\xi}||\mathbf{B}_{\xi}|\right)| \mathbf{W}_{ \xi \xi}|\text{d}\xi\\[0.5cm]
        & \leq C\delta \left(\lVert \mathbf{W}_{ \xi \xi}\rVert^2+\lVert \mathbf{W}_{ \xi \xi \xi}\rVert^2 + \lVert \mathbf{B}_{ \xi \xi}\rVert^2 +\lVert V_{ \xi \xi}\rVert^2+\lVert V_{\xi}\rVert^2\right).
    \end{aligned}
\end{equation}
\vspace{0.5cm}
\begin{equation}\label{eq71}
    \begin{aligned}
        I_{27}+I_{28} &=-\int \left(\dfrac{\nu}{\bar{v}^2}\right)_\xi \mathbf{B}_{\xi \xi}\, . \,\mathbf{B}_{\xi \xi \xi}\text{d}\xi +\int \left(\left(\dfrac{\nu}{\bar{v}^2}\right)_\xi \mathbf{B}_\xi\right)_{ \xi} \mathbf{B}_{\xi \xi \xi}\text{d}\xi\\[0.21cm]
        & \leq C \left(\lVert \mathbf{B}_{\xi}\rVert^2+\lVert \mathbf{B}_{\xi \xi}\rVert^2+\lVert \mathbf{B}_{\xi \xi \xi}\rVert^2\right)
    \end{aligned}
\end{equation}
Combining all together from \eqref{eq62}-\eqref{eq71} in \eqref{eq61}, choosing suitably small positive $\epsilon$ and  sufficiently positive small $\delta$, doing integration of resulting inequality with respect to $t$ and taking into account Lemmas \ref{lemma4.6}-\ref{lemma4.7}, \eqref{eq60} follows immediately. Hence, Lemma \ref{lemma4.8} proved.
\end{proof}
It now remains to control the term $\int_0^t\lVert H_{\xi \xi}(\tau)\rVert^2 \text{d}\tau$, for which we have
\begin{lemma}\label{lemma4.9}
    Suppose that the assumptions of the proposition \ref{proposition3.1} hold, then we have the following estimate for $t\in [0, T]$ as
    \begin{equation}\label{eq72}
        \begin{aligned}
            \int_0^t \lVert H_{\xi \xi}(\tau)\rVert^2 \text{d}\tau \leq C\lVert (V_0, H_0, \mathbf{W}_0, \mathbf{B}_0)\rVert_2^2 .
        \end{aligned}
    \end{equation}
\end{lemma}
\begin{proof}
    Multiplying $\eqref{eq26}_1$ by $H_{\xi \xi  \xi}$ and using $\eqref{eq26}_2$, we obtain
    \begin{equation}\label{eq73}
        \begin{aligned}
            H_{\xi \xi}^2 =& (V_\xi H_{\xi \xi})_t + \left(-V_tH_{\xi \xi}+ H_\xi H_{\xi \xi}+(f'(\bar{v})V_\xi)_\xi H_{\xi \xi}\right)_\xi\\[0.21cm]
            &-(f'(\bar{v})V_\xi)_{\xi \xi} H_{\xi \xi} - R_{2, \xi \xi}V_\xi + \left(p'(\bar{v})V_\xi\right)_{\xi \xi}V_\xi + R_1H_{\xi \xi \xi}
        \end{aligned}
    \end{equation}
    Now integrating \eqref{eq73} with respect to $\xi$ and $t$, we have
    \begin{equation}\label{eq74}
        \begin{aligned}
            \int_0^t \int H_{\xi \xi}^2\text{d}\xi \text{d}\tau =& \int \left(V_\xi H_{\xi \xi}- V_{0,\xi} H_{0,\xi \xi}\right)\text{d}\xi - \int_0^t \int R_{2,\xi \xi}V_\xi\text{d}\xi \text{d}\tau \\[0.21cm]
            & - \int_0^t \int (f'(\bar{v})V_\xi)_{\xi \xi} H_{\xi \xi}\text{d}\xi \text{d}\tau + \int_0^t \int R_1H_{\xi \xi \xi}\text{d}\xi \text{d}\tau + \int_0^t\int\left(p'(\bar{v})V_\xi\right)_{\xi \xi}V_\xi \text{d}\xi \text{d}\tau\\[0.21cm]
            & =: \sum_{i=29}^{33}I_i.
        \end{aligned}
    \end{equation}
    \vspace{0.5cm}
    Now we control the terms $I_i$ $(i\in\{29,..., 33\})$ as follows
    \begin{equation}\label{eq75}
        \begin{aligned}
            |I_{32}| &= \left|\int_0^t \int R_{1, \xi}H_{\xi \xi}\text{d}\xi \text{d}\tau\right|\\[0.21cm]
            & \leq C \int_0^t\int \left(V_{\xi\xi}^2 + |V_\xi||V_{\xi\xi}|^2 + |V_\xi||V_{\xi\xi\xi}| + |V_\xi|^2\right)|H_{\xi\xi}|\, d\xi d\tau\\[0.21cm]
            & \leq C \int_0^t \|V_{\xi\xi}\|_{L^\infty}\|V_{\xi\xi}\|\|H_{\xi\xi}\|\, d\tau + C \int_0^t \|V_\xi\|_{L^\infty}(\|V_\xi\|+\|V_{\xi\xi\xi}\|)\|H_{\xi\xi}\|\, d\tau\\[0.21cm]
            & \leq C \delta \int_0^t \Big((\|V_{\xi\xi}\|+\|V_{\xi\xi\xi}\|)\|V_{\xi\xi}\|+(\|V_\xi\|+\|V_{\xi\xi}\|)(\|V_\xi\|+\|V_{\xi\xi\xi}\|)\Big)\, d\tau  \\[0.21cm]
            & \leq C \delta \int_0^t \|(V_\xi,V_{\xi\xi},V_{\xi\xi\xi})\|^2\, d\tau, 
        \end{aligned}
    \end{equation}
    \begin{equation}\label{eq76}
        \begin{aligned}
            |I_{30}+I_{33}|=& \left|\int_0^t \int R_{2}V_{\xi \xi \xi}\text{d}\xi \text{d}\tau+\int_0^t\int p'(\bar{v})V_\xi V_{\xi \xi \xi} \text{d}\xi \text{d}\tau\right|\\[0.21cm]
            &\leq C \int_0^t \int \left( |V_\xi|^2+|\mathbf{B}_\xi|^2+|V_{ \xi}|\right)|V_{\xi \xi \xi}|\text{d}\xi \text{d}\tau\\[0.21cm]
            & \leq C\delta \int_0^t (\|V_\xi\|^2+\|V_{\xi \xi \xi}\|^2+\|\mathbf{B}_\xi\|^2)\text{d}\tau.
        \end{aligned}
    \end{equation}
    Meanwhile, a straightforward computation yields
    \begin{equation}\label{eq77}
        \begin{aligned}
            |I_{29}| \leq \|(V_\xi,H_{\xi\xi})\|^2 + \|(V_{0,\xi}, H_{0,\xi\xi})\|^2
        \end{aligned}
    \end{equation}
    and 
    \begin{equation}\label{eq78}
        \begin{aligned}
            |I_{31}| \leq \frac{1}{4}\int_0^t\|H_{\xi\xi}(\tau)\|^2 \, d\tau 
      + C \int_0^t \|(V_\xi,V_{\xi\xi},V_{\xi\xi\xi})\|^2 \, d\tau.
        \end{aligned}
    \end{equation}
    Plugging the estimates \eqref{eq75}-\eqref{eq78} in \eqref{eq74} and by virtue of Lemma \ref{lemma4.8}, \eqref{eq72} follows directly. This concludes the proof of Lemma \ref{lemma4.9}.
\end{proof}
It is observed that proof of Proposition \ref{proposition3.1} results as a consequence of Lemmas \ref{lemma4.5}-\ref{lemma4.9}. \hfill $\Box$
\section{Proof of Theorem \ref{theorem2.1}}\label{sec5}
In order to prove our main theorem \ref{theorem2.1}, it is enough to prove \eqref{eq21}. To prove this we follow the ideas of \cite{vasseur2016nonlinear}, \cite{he2020nonlinear}. With this first observe that from the global estimate \eqref{eq31} we have 
\begin{equation}\label{eq79}
    \begin{aligned}
        \|\left(V_\xi(.,t), H_\xi(.,t), \mathbf{W}_\xi(.,t), \mathbf{B}_\xi(.,t) \right)\|_1 \longrightarrow 0 \qquad \qquad \text{as}~~t\longrightarrow\infty.
    \end{aligned}
\end{equation}
Further from the fundamental theorem of calculus, 
\begin{equation}\label{eq80}
    \begin{aligned}
        V_\xi^2 &= 2\int_{-\infty}^\xi V_\xi V_{\xi \xi}(\eta,t)\text{d}\eta\\[0.21cm]
        & \leq 2 \|V_\xi(t)\|\|V_{\xi \xi}(t)\| \longrightarrow 0 \qquad \quad \text{as}~~t\longrightarrow\infty,
    \end{aligned}
\end{equation}
in a similar manner one can have
\begin{equation}\label{eq81}
    \begin{aligned}
        H_\xi(\xi, t), \mathbf{W}_\xi(\xi, t), \mathbf{B}_\xi(\xi, t) \longrightarrow 0 \qquad \quad \text{as}~~t\longrightarrow\infty, ~\forall \xi \in \mathbb{R}^1
    \end{aligned}
\end{equation}
It is easy to see that
\begin{equation}\label{eq82}
    \begin{aligned}
        \|\mathbf{b}\|_{L^\infty}(t)&\leq C \|\mathbf{B}_\xi\|^{\frac{1}{2}}\|\mathbf{B}_{\xi \xi}\|^{\frac{1}{2}}\\[0.21cm]
        & \leq C \|\mathbf{B}_\xi\|^{\frac{1}{2}}\longrightarrow 0, \quad \quad \text{as}~~t\longrightarrow\infty.
    \end{aligned}
\end{equation}
Now we need to prove $\sup_{\xi \in \mathbb{R}}|u(\xi,t)-\bar{u}(\xi)|\longrightarrow 0 ~~\text{as}~~t\longrightarrow\infty$, for this we have from $\eqref{eq22}_2$ as
\begin{equation}\label{eq83}
    \begin{aligned}
        u_t - \sigma u_\xi+ \left(p(v)+\dfrac{1}{2}|\mathbf{b}|^2\right)_\xi = \left(f'(v)u_\xi\right)_\xi
    \end{aligned}
\end{equation}
and for travelling wave,
\begin{equation}\label{eq84}
    \bar{u}_t - \sigma \bar{u}_\xi + p(\bar{v})_\xi = \left(f'(\bar{v})\bar{u}_\xi\right)_\xi
\end{equation}
subtracting \eqref{eq84} from \eqref{eq83} and define $z:= u-\bar{u}$ we obtain
\begin{equation}\label{eq85}
    \begin{aligned}
        z_t - \left(f'(v)z_\xi\right)_\xi = \sigma z_\xi - \left(p(v)-p(\bar{v})\right)_\xi + \dfrac{1}{2}\left(|\mathbf{b}|^2\right)_\xi + \left(\left(f'(v)-f'(\bar{v})\right)\bar{u}_\xi\right)_\xi.
    \end{aligned}
\end{equation}
From the global estimate \eqref{eq31}, it is evident that the right hand side of \eqref{eq85} is bounded in $L^2\left(0,T; L^2(\mathbb{R})\right)$. From the regularity results for parabolic equations of the form \eqref{eq85}, we have
\begin{equation}\label{eq86}
    \begin{aligned}
        z_t ~~\text{is bounded in}~~L^2(0,T; L^2(\mathbb{R})),
    \end{aligned}
\end{equation}
and 
\begin{equation}\label{eq87}
    \begin{aligned}
        z ~~\text{is bounded in}~~L^2(0,T; H^2(\mathbb{R})),
    \end{aligned}
\end{equation}
which in turn
\begin{equation}\label{eq88}
    \begin{aligned}
        \|z(\xi, t)\| \longrightarrow 0, \qquad \quad \text{as}~~t\longrightarrow\infty.
    \end{aligned}
\end{equation}
Further set $Z = z_\xi$, which satisfy
\begin{equation}\label{eq89}
    \begin{aligned}
        Z_t - \left(f'(v)Z\right)_{\xi \xi} = \sigma Z_\xi - \left(p(v)-p(\bar{v})\right)_{\xi \xi} + \dfrac{1}{2}\left(|\mathbf{b}|^2\right)_{\xi \xi} + \left(\left(f'(v)-f'(\bar{v})\right)\bar{u}_\xi\right)_{\xi \xi}.
    \end{aligned}
\end{equation}
By \eqref{eq31} and \eqref{eq87}, again the right hand side of \eqref{eq89} is bounded in  $L^2\left(0,T; L^2(\mathbb{R})\right)$, then from regularity result for parabolic equations with initial data
\begin{equation}\label{eq90}
    \begin{aligned}
       Z |_{t=t_0}  = z_\xi(.,t_0)
    \end{aligned}
\end{equation}
which is bounded in $L^2(\mathbb{R})$, subsequently yields
\begin{equation}\label{eq91}
    \begin{aligned}
        Z_t ~~\text{is bounded in}~~L^2(t_0,T; L^2(\mathbb{R}))
    \end{aligned}
\end{equation}
and 
\begin{equation}\label{eq92}
    \begin{aligned}
        Z ~~\text{is bounded in}~~L^2(t_0,T; H^2(\mathbb{R})),
    \end{aligned}
\end{equation}
which in turn
\begin{equation}\label{eq93}
    \begin{aligned}
        \|Z(\xi, t)\| \longrightarrow 0, \qquad \quad \text{as}~~t\longrightarrow\infty.
    \end{aligned}
\end{equation}
From \eqref{eq88} and \eqref{eq93} we have
\begin{equation}\label{eq94}
    \begin{aligned}
        \|z\|_1 \longrightarrow 0, \qquad \quad \text{as}~~t\longrightarrow\infty
    \end{aligned}
\end{equation}
and by the Sobolev inequality
\begin{equation}\label{eq95}
    \begin{aligned}
        \sup_{\xi \in \mathbb{R}}|z(\xi, t)| \longrightarrow 0, \quad \quad \text{as}~~t\longrightarrow\infty.
    \end{aligned}
\end{equation}
Therefore from \eqref{eq80}-\eqref{eq82}, \eqref{eq95} we deduce
\begin{align}\label{eq96}
    \sup_{\xi\in \mathbb{R}}\left|(v-\overline{v}, u-\overline{u}, \mathbf{w}, \mathbf{b})\right| \longrightarrow 0 \quad \text{as}~~t\longrightarrow \infty.
\end{align}
Which completes the proof of Theorem \ref{theorem2.1}.
\section*{Acknowledgments}
The first author expresses gratitude for financial support from the Indian Institute of Technology (Ref. No. IIT/ACAD/RS/EN/23MA91R03/2). The second author is grateful for the funding provided by the core research grant from ANRF, DST, Government of India (Ref. No. CRG/2022/006297).
\bibliographystyle{elsarticle-num}
	\bibliography{bibmarko}



\end{document}